\documentclass[12pt]{amsart}
\usepackage{amsmath,latexsym,amscd,amsbsy,amssymb,amsfonts,amsthm,fleqn,leqno,
euscript, graphicx, texdraw, pb-diagram}
\usepackage{mathrsfs}
\usepackage[matrix,arrow,curve]{xy}

\newtheorem{thm}{Theorem}[section]
\newtheorem{pro}[thm]{Proposition}
\newtheorem{lem}[thm]{Lemma}

\newtheorem{cor}[thm]{Corollary}

\def\leukfrac#1/#2{\leavevmode
               \kern.1em
                \raise.9ex\hbox{\the\scriptfont0 ${}_#1$}
                \hskip -1pt\kern-.1em
                /\kern-.15em\lower.10ex\hbox{\the\scriptfont0 ${}_#2$}}

\theoremstyle{definition}

\theoremstyle{remark}

\makeatletter
\def\Int{\mathop{\operator@font Int}\nolimits}
\makeatother

\begin{document}

\title[Proper absolute extensors]
{Proper absolute extensors}

\author{Vesko  Valov}
\address{Department of Computer Science and Mathematics, Nipissing University,
100 College Drive, P.O. Box 5002, North Bay, ON, P1B 8L7, Canada}
\email{veskov@nipissingu.ca}
\thanks{The second author was partially by NSERC Grant 261914-19}




\begin{abstract}
We describe the proper absolute (neighborhood) extensors for the class of at most $n$-dimensional spaces, notation $\rm{A(N)E}_p(n)$.
For example, the unique locally compact $n$-dimensional separable metric space $X\in\rm{AE}_p(n)$ satisfiyng the $\rm{DD^nP}$-property is the $n$-dimensional Menger compactum without a point. Non-metrizable $\rm{A(N)E}_p(n)$-spaces are also described.
\end{abstract}

\maketitle

\markboth{}{Proper extensors}



\section{Introduction and preliminary results}
In this note we describe the proper absolute extensors for finite-dimensional spaces, see Theorem 2.4 and Theorem 3.2. Recall that a map $f: X\to Y$ is proper if $f^{-1}(K)$ is compact for every compact $K\subset Y$. Note that if $X, Y$ are locally compact, then $f$ is proper iff it is closed and all fibres $f^{-1}(y)$, $y\in Y$, are compact. Proper and closed extensions of maps were considered by different authors, see Michael \cite{m}, Nowi\'{n}ski \cite{n}. Our results are closer to Chigogidze's ones from \cite{chi}, where proper absolute
(neighborhood) extensors were introduced and studied.

We say that a locally compact space $X$ is a {\em proper absolute neighborhood extensor for the class of at most $n$-dimensional spaces} (notation $X\in\rm{ANE}_p(n)$) if every proper map $f:A\to X$, where $A$ is a closed subset of a locally compact Lindel\"{o}f-space $Y$ with $\dim Y\leq n$, admits a perfect extension $\widetilde{f}$ over a closed neighborhood of $A$ in $Y$. When $f$ admits a proper extension over $Y$, we say $X$ is
a {\em proper absolute extensor for the class of at most $n$-dimensional spaces} (notation $X\in\rm{AE}_p(n)$).
Since every space admitting a proper map into a compact space is compact, it follows from the definition that there is no compact $\rm{AE}_p(n)$-space. In particular, the $n$-sphere $\mathbb S^n$, which is an absolute extensor for the $n$-dimensional spaces, is not an $\rm{AE}_p(n)$.

If, in the above definition, $Y$ is metric and we drop the requirement for $f$ and $\widetilde f$ to be proper maps, we obtain the definition of absolute (neighborhood) extensors (notation $\rm{A(N)E}(n)$) for the class of at most $n$-dimensional spaces. It is well known \cite{bo} that a metric space $X$ is an $\rm{ANE}(n)$ iff $X$ is $\rm{LC}^{n-1}$. Moreover, if in addition, $X$ is $\rm C^{n-1}$, then $X\in\rm{AE}(n)$.
Recall that a space $X$ is $\rm{LC}^n$ if for every $x\in X$ and its neighborhood $U$ in $X$ there is another neighborhood $V$ of $x$ such that $V\stackrel{m}{\hookrightarrow}U$ for all $m\leq n$ (here $V\stackrel{m}{\hookrightarrow}U$ means that $V\subset U$ and every map from $m$-dimensional sphere $\mathbb S^m$ into $V$ can be extended to a map $\mathbb B^{m+1}\to U$ over the $(m+1)$-dimensional cub $\mathbb B^{m+1}$). We also say that a set $A\subset X$ is
$k-\rm{LCC}$ in $X$ if for every point $x\in A$ and its neighborhood $U$ in $X$ there exists another neighborhood $V$ of $x$ with $V\setminus A\stackrel{k}{\hookrightarrow}U\setminus A$. If $A$ is $k-\rm{LCC}$ in $X$ for all $k\leq n$, then $A$ is said to be $\rm{LCC}^n$ in $X$.

\section{Second countable $\rm{AE}_p(n)$-spaces}

Everywhere in this section by a space, if not explicitly said otherwise, we mean a locally compact separable metric space. By $C(Z,X)$ we denote the set of all continuous maps from $Z$ to $X$ equipped with the compact-open topology.
A close subset $A\subset X$ is said to be a {\em $Z_n$-set in $X$} \cite{vm} if the set $C(\mathbb B^{n}, X\backslash A$ is dense in
$C(\mathbb B^{n},X)$. The following description of $Z_n$-sets in metric spaces is well known, but we couldn't find a reference.
\begin{lem}\label{z-sets}
Let $(X,d)$ be a metric $\rm{LC}^{n-1}$-space and $A$ be a closed nowhere dense set in $X$. Then $A$ is a  $Z_n$-set in $X$ iff $A$ is $\rm{LCC}^{n-1}$.
\end{lem}

\begin{proof}
The sufficiency follows from the properties of metric $\rm{LC}^{n-1}$-spaces and the definition of $Z_n$-sets.
 Suppose $A$ is $\rm{LCC}^{n-1}$,
$f:\mathbb B^n\to X$ is a given map and $\eta>0$. We consider the following property for every $x\in X$: if $U$ is a neighborhood of $x$ in $X$, then there is another neighborhood $V\subset U$ such that $V\setminus A\stackrel{n-1}{\hookrightarrow}U\setminus A$. Because $A$ is
 $\rm{LCC}^{n-1}$ and $X$ is $\rm{LC}^{n-1}$, every $x\in X$ has that property. So, by \cite{dm} we can assume that the metric $d$ satisfies the following condition: To every $\varepsilon>0$ there corresponds a $\delta>0$ such that $B_\delta(x)\setminus A\stackrel{m}{\hookrightarrow} B_\varepsilon(x)\setminus A$ for every $x\in X$ and every $m\leq n-1$. Here $B_\delta(x)$ denotes the open ball in $X$ with center $x$ and a radius $\delta$.
We write
$\delta\stackrel{m}{\hookrightarrow}\varepsilon$ to denote that $B_\delta(x)\setminus A\stackrel{m}{\hookrightarrow} B_\varepsilon(x)\setminus A$ for every $x\in X$. We choose a finite sequence $\{\varepsilon_m\}_{m\leq n}$ with $\varepsilon_m\stackrel{m}{\hookrightarrow}\varepsilon_{m+1}$ for every $0\leq m\leq n-1$ and $\varepsilon_n=\eta$. Next, define the set-valued maps $\varphi_m:\mathbb B^n\rightsquigarrow X$ by $\displaystyle\varphi_m(y)=B_{\varepsilon_m}(f(y))\setminus A$, $0\leq m\leq n$. Since $A$ is $\rm{LCC}^{n-1}$, $\varphi_m(y)\neq\varnothing$ for all $y\in\mathbb B^n$.
It is easily seen that each $\varphi_m$ has the following property: If $K$ is a compact subset of $\varphi_m(y_0)$ for some $y_0\in\mathbb B^n$, then there is a neighborhood $O(y_0)\subset\mathbb B^k$ with $K\subset\varphi_m(y)$ for all $y\in O(y_0)$. According to
\cite{g}, there exists a map $g:\mathbb B^n\to X$ with $g(y)\in\varphi_n(y)$ for all $y\in\mathbb B^n$. This means that $g$ maps $\mathbb B^n$ into $X\setminus A$ and $d(f(y),g(y))<\eta$ for every $y\in\mathbb B^n$. Hence, $A$ is a  $Z_n$-set in $X$.
\end{proof}

For every locally compact space $X$ let $\omega X=X\cup\{\omega\}$ be the one-point compactification of $X$.
\begin{lem}
If $X$ is an $\rm{AE}_p(n)$-space, then $\{\omega\}$ is a $Z_n$-set in $\omega X$.
\end{lem}
\begin{proof}
Since every metric $\rm{AE}_p(n)$-space is an absolute extensor for compact spaces of dimension $\leq n$, $X$ is $\rm{LC}^{n-1}$
Hence, by Lemma 2.1, we need to show that $\{\omega\}$ is $\rm{LCC}^{n-1}$. Suppose there are $k\leq n-1$ and a neighborhood $U$ of $\{\omega\}$ in $\omega X$ such that for every neighborhood $V$ of $\{\omega\}$ there exists a map $g_V:\mathbb S^k\to V\backslash\{\omega\}$ which does not admit an extension from $\mathbb B^{k+1}$ into $U\backslash\{\omega\}$. Take a local base $\{V_m\}$ of neighborhoods of $\{\omega\}$ in $\omega X$ and corresponding maps $g_m:\mathbb S^k\to V_m\backslash\{\omega\}$ such that $V_m\subset U$ and each $g_m$ cannot be extended to a map $\widetilde g_m:\mathbb B^{k+1}\to U\backslash\{\omega\}$. Now, for each $m$ let $\mathbb S^k_m$ and $\mathbb B^{k+1}_m$ be copies of
$\mathbb S^k$ and $\mathbb B^{k+1}$, respectively. Consider the disjoint unions $Y=\biguplus_{m=1}^\infty\mathbb B^{k+1}_m$,
$A=\biguplus_{m=1}^\infty\mathbb S^{k}_m$ and their one-point compactification $\omega Y=Y\cup\{\omega_Y\}$. Obviously,
$\omega A=A\cup\{\omega_Y\}$ and there is a map $g: A\to X$ with $g|\mathbb S^{k}_m=g_m$. Since the map $g$ is proper, it admits a proper extension $\widetilde g:Y\to X$. Hence, $\widetilde g$ is extended to a map $h:\omega Y\to \omega X$
such that $h(\{\omega_Y\})=\{\omega\}$. Consequently, $h^{-1}(U)$ contains almost all $\mathbb B^{k+1}_m$. On the other hand, $h(\mathbb B^{k+1}_m)=\widetilde g(\mathbb B^{k+1}_m)\subset X$.
Therefore, $\widetilde g|\mathbb B^{k+1}_m)$ is a map into $U\backslash\{\omega\}$ extending $g_m$ for every $\mathbb B^{k+1}_m$ contained in
$h^{-1}(U)$, a contradiction.
\end{proof}

\begin{pro}
Every space $X$  is an $\rm{AE}_p(0)$.
\end{pro}

\begin{proof}
Let $A\subset Y$ be a closed set and $f:A\to X$ be a proper map, where $Y$ is a $0$-dimensional locally compact and Lindel\"{o}f space.
Then $f$ can be extended to a map $f_1:\overline A\to\omega X$ over the closure of $A$ in $\beta Y$ with $f_1(\overline A\backslash A)=\{\omega\}$.
Next, consider the map $g:A\cup (\beta Y\backslash Y)\to \omega X$ such that $g(y)=f_1(y)$ for $y\in\overline A$ and $g(y)=\{\omega\}$ for
$y\in\beta Y\backslash Y$. Since $\omega X\in\rm{AE}(0)$ (as a complete metric space), $g$ admits an extension $\widetilde g:\beta Y\to\omega X$.
The set $\widetilde g(A)=f(A)$ is closed in $X$, so $A_1=\widetilde g^{-1}(f(A))$ is closed and $G_\delta$ in $Y$ such that
$\widetilde g|A_1$ is proper. Consider the function space $C(\beta Y,\omega X)$ with the uniform convergence topology and let $B=A_1\cup (\beta Y\backslash Y)$. The set $C_B=\{h\in C(\beta Y,\omega X): h|B=\widetilde g|B\}$ is a complete metric space.
Choose a sequence
$\{K_i\}$ of compact sets $K_i\subset Y$ with
$\bigcup_{i\geq 1}K_i=Y\backslash A_1$ (this is possible because $Y$ is a locally compact Lindel\"{o}f-space and $A_1$ is a closed $G_\delta$-set in $Y$).
Let $C_i$ be the set of all maps $h\in C_B$ such that $h(K_i)\subset X$. Since $\{\omega\}$ is nowhere dense set in $\omega X$ (it is actually a $Z_0$-set in $\omega X$) and $\omega X\in\rm{AE}(0)$, we can show that each $C_i$ is an open and dense subset of $C_B$.  Hence, $\bigcap_{i\geq 1}C_i\neq\varnothing$. Then $h(Y)\subset X$ and
$h(\beta Y\backslash Y)=\{\omega\}$ for every $h\in\bigcap_{i\geq 1}C_i$. Hence, $h|Y$ is a proper map into $X$ extending $f$.
\end{proof}

\begin{thm}
The following conditions are equivalent for any space $X$ and $n\geq 1$:
\begin{itemize}
\item[(1)] $X\in\rm{AE}_p(n)$;
\item[(ii)] The one-point compactification $\omega X$ is an $\rm{AE}(n)$ and $\{\omega\}$ is a $Z_n$-set in $\omega X$;
\item[(iii)] There exists a metrizable compactification $\widetilde X$ of $X$ such that both $\widetilde X$ and the remainder $\widetilde X\backslash X$ are $\rm{AE}(n)$spaces, and
$\widetilde X\backslash X$ is an $Z_n$-set in $\widetilde X$.
\end{itemize}
\end{thm}

\begin{proof}
Suppose $X\in\rm{AE}_p(n)$  and embed
$\omega X$ in the Hilbert cube $Q$. According to Dranishnikov \cite{dr} there exists a surjective, open $n$-invertible map $d_n:\mu^n\to Q$
such that $d_n^{-1}(z)$ is homeomorphic to $\mu^n$ for every $z\in Q$,
where $\mu^n$ is the universal $n$-dimensional Menger compactum. Recall that the $n$-invertibility of $d_n$ means that for any paracompact space $Z$ of dimension $\dim Z\leq n$ and a map $g:Z\to Q$ there is a map $\widetilde g:Z\to\mu^n$ such that $d_n\circ\widetilde g=g$.
Then $d_n^{-1}(\{\omega\})$ is nowhere dense in $\mu^n$ and
$\mu^n$ is a compactification of $\mu^n\backslash d_n^{-1}(\{\omega\})$.
Now, consider the restriction $d_n'=d_n|d_n^{-1}(X)$. Obviously, $d_n'$ is a proper map, so it admits a proper extension $h_n:\mu^n\backslash d_n^{-1}(\{\omega\})\to X$. The properness of $h_n$ implies that $h_n$ can be extended to a continuous map $\widetilde h_n:\mu^n\to\omega X$ such that $\widetilde h_n(d_n^{-1}(\{\omega\}))=\{\omega\}$. Then
$\widetilde h_n|(d_n^{-1}(\omega X))=d_n|(d_n^{-1}(\omega X))$. Hence, $\widetilde h_n$ is an $n$-invertible map because so is $d_n$. This fact in combination with $\mu^n\in\rm{AE}(n)$ yields that $\omega X\in\rm{AE}(n)$. Finally, by Lemma 2.2, $\{\omega\}$ is a $Z_n$-set in $\omega X$. That completes the implication $(i)\Rightarrow (ii)$.
The implication $(ii)\Rightarrow (iii)$ is trivial.

To prove the implication $(iii)\Rightarrow (i)$ we follow the proof of Proposition 2.3. Let $A\subset Y$ be a closed set and $f:A\to X$ be a proper map, where $Y$ is at most $n$-dimensional locally compact and Lindel\"{o}f space. Following the notations from the proof of Proposition 2.3, we first extend $f$ to a map $f_1:\overline A\to\widetilde X$ with $f_1(\overline A\backslash A)\subset\widetilde X\backslash X$. Then, using that $\widetilde X\backslash X\in\rm{AE}(n)$, we extend $f_1$ to a map $g:\overline A\cup(\beta Y\backslash Y)\to\widetilde X$ such that
$g(\beta Y\backslash Y)\subset\widetilde X\backslash X$. Next, since
$\widetilde X$ is an $\rm{AE}(n)$, we find a map $\widetilde g:\beta Y\to\widetilde X$ extending $g$. Then $\widetilde g|A_1:A_1\to X$ is a proper extension of $f$ with $A_1\subset Y$ being a closed $G_\delta$-subset of $Y$ containing $A$. Let
$B=A_1\cup (\beta Y\backslash Y)$  and consider a sequence $\{K_i\}$ of compact sets in $Y$ with $\bigcup_{i\geq 1}K_i=Y\backslash A_1$ and the corresponding sets $C_B=\{h\in C(\beta Y,\widetilde X): h|B=\widetilde g|B\}$ and $C_i$. Now, since $\widetilde X\backslash X$ is $Z_n$-set in $\widetilde X$, all $C_i$ are open and dense in $C_B$.
Therefore, $\bigcap_{i\geq 1}C_i\neq\varnothing$ and $h|Y$ is a proper map into $X$ extending $f$ for every $h\in\bigcap_{i\geq 1}C_i$.
\end{proof}

We say that a space $X$ satisfies the {\em disjoint $n$-disks property} (br., $\rm{DD^nP}$-property) if any two maps $f,g:\mathbb B^n\to X$ can be approximated by maps $f',g':\mathbb B^n\to X$ with $f'(\mathbb B^n)\cap g'(\mathbb B^n)=\varnothing$. Bestvina \cite{b} characterized $\mu^n$ as the only $n$-dimensional metric $\rm{AE}(n)$-compactum satisfying the $\rm{DD^nP}$-property. Since $\omega X$ satisfies the $\rm{DD^nP}$ provided $X\in \rm{DD^nP}$ and $\{\omega\}$ is a $Z_n$-set in $\omega X$, Bestvina's result implies the following one:
\begin{cor}
Let $X\in\rm{AE}_p(n)$ with $\dim X=n$. Then $X\in\rm{DD^nP}$ iff $\omega X$ is homeomorphic to $\mu^n$.
\end{cor}

Chigogidze \cite{chi1} introduced the $n$-shape functor ($n-\rm{Sh}$) and proved that two $Z_n$-sets $X$ and $Y$ in $\mu^n$, $n\geq 1$, have the same
$(n-1)$-shape if and only $\mu^{n}\backslash X$ is homeomorphic to $\mu^{n}\backslash Y$.
Surprisingly, Theorem 2.4 implies a particular case of Chigogidze's complement theorem.

\begin{cor}
Suppose $X$ and $Y$ are two $Z_n$-sets in $\mu^n$ such that $X, Y\in\rm{AE}(n)$. Then $\mu^{n}\backslash X$ is homeomorphic to $\mu^{n}\backslash Y$.
\end{cor}
\begin{proof}
Indeed, by Theorem 2.4 both $X'=\mu^{n}\backslash X$ and $Y'=\mu^{n}\backslash Y$ are $\rm{AE}_p(n)$. Moreover, $X'$ and $Y'$ satisfy the $\rm{DD^nP}$ because $X$ and $Y$ are $Z_n$-sets in $\mu^n$. Hence, by Corollary 2.5, both $\omega X'$ and $\omega Y'$ are homeomorphic to $\mu^n$. Finally, since $\mu^n$ is homogeneous \cite{b}, $X'$ is homeomorphic to $Y'$.
\end{proof}

\begin{pro}
A space $X$ is an $\rm{AE}_p(n)$ if and only if $X$ is a proper $n$-invertible image of $\mu^n\backslash F$ for some $Z_n$-set $F\subset\mu^n$ with $F\in\rm{AE}(n)$.
\end{pro}

\begin{proof}
Let $X\in\rm{AE}_p(n)$ and embed $\omega X$ in $Q$. As in Theorem 2.4, considering Dranishnikov's resolution $d_n:\mu^n\to Q$ we obtain a proper map $h_n:\mu^n\backslash d_n^{-1}(\{\omega\})\to X$ which extends the map $d_n|d_n^{-1}(X)$. Since $d_n$ is invertible, so is $h_n$. On the other hand, by \cite{acr}, we can assume that $d_n^{-1}(K)$ is a $Z_n$-set in $\mu^n$ for every $Z$-set $K\subset Q$.
Hence, $d_n^{-1}(\{\omega\})\subset\mu^n$ is a $Z_n$-set (recall that a $Z$-set is a set which is $Z_n$-set for all $n$ and that every $z\in Q$ is a $Z$-set in $Q$). On the other hand, $d_n^{-1}(\{\omega\})$ is homeomorphic to $\mu^n$, so $d_n^{-1}(\{\omega\})\in\rm{AE}(n)$.

Now, suppose there is a proper $n$-invertible map $g: \mu^n\backslash F\to X$ for some $Z_n$-set $F\subset\mu^n$. Since $g$ is $n$-invertible, every proper map $f:A\to X$, where $A$ is closed subset of at most $n$-dimensional locally compact and Lindel\"{o}f space $Y$, can be lifted to a proper map $f':A\to \mu^n\backslash F$. By Theorem 2.4, $\mu^n\backslash F\in\rm{AE}_p(n)$. So, $f'$ admits a proper extension
$h:Y\to\mu^n\backslash F$. Finally, $g\circ h:Y\to X$ is a proper extension of $f$. Therefore, $X\in\rm{AE}_p(n)$.
\end{proof}

\begin{cor}
A space $X$ is an $\rm{AE}_p(n)$ if and only if $X$ is a proper $n$-invertible image of $\mu^n\backslash\{pt\}$.
\end{cor}

\begin{proof}
By Theorem 2.4, $\mu^n\backslash F$ is an $\rm{AE}_p(n)$ for every $F\in\rm{AE}(n)$ which is a $Z_n$-set in $\mu^n$. On the other hand,
$\mu^n\backslash F$ satisfies the $\rm{DD^nP}$ as a complement of a $Z_n$-set in $\mu^n$. Hence,
Proposition 2.7 and Corollary 2.5 complete the proof.
\end{proof}

Concerning $\rm{ANE}_p(n)$, arguments similar to the proof of Proposition 2.3 provide the next lemma.
\begin{lem}
If a space $X$ admits a metric $\rm{ANE}(n)$-compactification $\overline X$ such that $\overline X\backslash X$ is an $\rm{AE}(n)$, then $X\in \rm{ANE}_p(n)$.
\end{lem}

\begin{cor}
$\mathbb R^n$ is an  $\rm{ANE}_p(n)$ and an $\rm{AE}_p(n-1)$, but not an $\rm{AE}_p(n)-space$.
\end{cor}

\begin{proof}
It follows from Theorem 2.4 that $\mathbb R^n$ is an $\rm{AE}_p(n-1)$, but not an $\rm{AE}_p(n)$-space because $\omega\mathbb R^n=\mathbb S^n$ and any point of $\mathbb S^n$ is a $Z_{n-1}$-point in $\mathbb S^n$ but not a $Z_n$-point. On the other hand, $\mathbb R^n\in\rm{ANE}_p(n)$ according to Lemma 2.9.
\end{proof}


\section{Non-metrizable $\rm{AE}_p(n)$-spaces}
In this section all spaces are locally compact and Lindel\"{o}f.
A map $f:X\to Y$ is called {\em $n$-soft} \cite{sc} if for every $n$-dimensional paracompact space $Z$, any closed set $A\subset Z$ and any two maps $h:A\to X$ and $g:Z\to Y$ with $g|A=f\circ h$ there is a continuous extension $\widetilde h:Z\to X$ of $h$ such that $g=f\circ\widetilde h$.
We say that $f:X\to Y$ is a {\em map with a Polish kernel} if there is a Polish (i.e., completely metrizable and separable) space $P$ such that $X$ is $C$-embedded in $Y\times P$ and $f=\pi_Y|X$, where $\pi_Y:Y\times P\to Y$ is the projection.

The next lemma follows from the corresponding definitions.
\begin{lem}
Let $f:X\to Y$ is a proper $n$-soft map . Then $X\in\rm{AE}_p(n)$ if and only if $Y\in\rm{AE}_p(n)$.
\end{lem}

An inverse system $\displaystyle S=\{X_\alpha, p^{\beta}_\alpha, A\}$ is said to be {\em $\sigma$-complete} if all $X_\alpha$ are second countable spaces and every increasing sequence $\{\alpha_n\}\subset A$ has a supremum $\alpha$ in $A$ such that $X_\alpha$ is the limit space of the inverse sequence $\displaystyle \{X_{\alpha_n}, p^{\alpha_{n+1}}_{\alpha_n}, n\geq 1\}$. If $S$ is well-ordered and $X_\alpha$ is the limit of the inverse system $\displaystyle \{X_\beta, p^{\beta+1}_\beta, \beta<\alpha\}$ for every limit ordinal $\alpha\in A$, then $S$ is called a {\em continuous inverse system}.

Now, we can describe the non-metrizable $\rm{AE}_p(n)$-spaces.

\begin{thm}
For every $n\geq 1$ the following conditions are equivalent :
\begin{itemize}
\item[(i)] $X$ is an  $\rm{AE}_p(n)$-space of weight $\tau$;
\item[(ii)] $\omega X\in\rm{AE}(n)$;
\item[(iii)] $X$ is the limit space of a continuous inverse system  $\displaystyle S=\{X_\alpha, p^{\beta}_\alpha, \tau\}$ such that
all $X_\alpha$ are $\rm{AE}_p(n)$-spaces, $X_1$ is a locally compact separable metric space and the projections $p^{\alpha+1}_\alpha$ are perfect $n$-soft maps with metrizable kernels;
\item[(iv)] $X$ is the limit space of a $\sigma$-complete inverse system  $\displaystyle S=\{X_\alpha, p^{\beta}_\alpha\}$ consisting of
$\rm{AE}_p(n)$-spaces $X_\alpha$ and perfect $n$-soft projections $p_\alpha:X\to X_\alpha$.
\end{itemize}
\end{thm}
\begin{proof}
Let $X$ be an $\rm{AE}_p(n)$-space of weight $\tau$ and embed $\omega X$ in the Tychonoff cube $\mathbb I^\tau$. According to \cite{dr1}, there exists a compact $AE(n-1)$-space $D_n^\tau$ of dimension $n$ and weight $\tau$ and an $n$-invertible $(n-1)$-soft-map $f_n^\tau:D_n^\tau\to\mathbb I^\tau$. Since $\{\omega\}$ is a $G_\delta$-set in $\omega X$, there is a closed $G_\delta$-set $F\subset\mathbb I^\tau$ with $F\cap\omega X=\{\omega\}$. Deleting the interior of $F$, if necessary, we can suppose that $F$ is nowhere dense in $\mathbb I^\tau$. Then $(f_n^\tau)^{-1}(F)$ is a
closed nowhere dense and $G_\delta$-subset of $D_n^\tau$ because $f_n^\tau$ is open (as a $0$-soft map between $\rm{AE}(0)$-spaces, see \cite{chi1}). So,
 $Y=D_n^\tau\backslash (f_n^\tau)^{-1}(F)$ is a dense locally compact Lindel\"{o}f subset of $D_n^\tau$ containing $(f_n^\tau)^{-1}(X)$ as a closed subset. Since $X\in\rm{AE}_p(n)$, there is a proper map $g:Y\to X$ extending the restriction $f_n^\tau|(f_n^\tau)^{-1}(X)$. Finally, extend $g$ to a map $\widetilde g:D_n^\tau\to\omega X$. Now, consider the set valued map $r:\mathbb I^\tau\rightsquigarrow\omega X$, $r(x)=\widetilde g((f_n\tau)^{-1}(x))$. Obviously, $r(x)=\{x\}$ for every $x\in\omega X$. Since $f_n^\tau$ is $(n-1)$-soft, $r$ is projectively $(n-1)$-cosoft retraction in the sense of Dranishnikov \cite{dr1}. Hence, by \cite[Theorem 4.2]{dr1}, $\omega X$ is an $\rm{AE}(n)$-space. So, $(i)\Rightarrow (ii)$.

 If $\omega X\in\rm{AE}(n)$, then $\omega X$ is the limit a continuous inverse system  $\displaystyle\widetilde S=\{\widetilde X_\alpha, \widetilde p^{\beta}_\alpha, \tau\}$ such that $\widetilde X_1$ is a point and all projections $\widetilde p^{\alpha+1}_\alpha$ are $n$-soft maps with metrizable kernels, see \cite[Theorem 4.2]{dr1}. Because $\{\omega\}$ is a $G_\delta$-set in $\omega X$, there is $\alpha_0<\tau$ such that $\widetilde X_{\alpha_0}$ is metrizable and
 $\widetilde p_{\alpha_0}^{-1}(\widetilde p_{\alpha_0}(\{\omega\}))=\{\omega\}$. Consequently, $\widetilde p_{\alpha}^{-1}(X_\alpha)=X$ for every $\alpha\geq\alpha_0$, where $X_\alpha=\widetilde X_{\alpha}\backslash\widetilde p_{\alpha}(\{\omega\})$. Obviously, all restrictions
 $p_\alpha=\widetilde p_\alpha|X$ and $p^{\beta}_\alpha=\widetilde p^{\beta}_\alpha|X_\beta$ are perfect $n$-soft maps and $X$ is the limit of the inverse system
  $\displaystyle S=\{X_\alpha, p^{\beta}_\alpha, \alpha\geq\alpha_0\}$. Finally, by Lemma 3.1, each $X_\alpha$ is an $\rm{AE}_p(n)$. This complete the implication $(ii)\Rightarrow (iii)$.

  The implication $(iii)\Rightarrow (iv)$ follows by similar arguments using that $\omega X$ (as an $\rm{AE}(n)$-compactum) is the limit space of
a $\sigma$-complete inverse system  $\displaystyle\widetilde S=\{\widetilde X_\alpha, \widetilde p^{\beta}_\alpha\}$ consisting of
$\rm{AE}(n)$-metric compacta $\widetilde X_\alpha$ and perfect $n$-soft projections $\widetilde p_\alpha:X\to X_\alpha$, see \cite[Theorem 4.2]{dr1}.
Finally, since $X$ admits $n$-soft perfect maps into $\rm{AE}_p(n)$-spaces, the implication $(iv)\Rightarrow (i)$ follows from Lemma 3.1.
\end{proof}
Because every $\rm{AE}(n)$-compactum of dimension $\leq n$, where $n\geq 1$, is metrizable \cite[Theorem 4.4]{dr1}, Theorem 3.2 implies the following
\begin{cor}
Every $\rm{AE}_p(n)$-space $X$ with $n\geq 1$ is metrizable provided $\dim X\leq n$.
\end{cor}

Concerning $\rm{AE}_p(0)$-spaces we have the following:
\begin{thm}
The following conditions are equivalent:
\begin{itemize}
\item[(i)] $X$ is an $\rm{AE}_p(0)$-space;
\item[(ii)] $X\in\rm{AE}(0)$;
\item[(iii)] $\omega X\in\rm{AE}(0)$;
\item[(iv)] $X$ is the limit space of a  $\sigma$-complete inverse system  $\displaystyle S=\{X_\alpha, p^{\beta}_\alpha\}$ consisting of
locally compact separable metric spaces $X_\alpha$ and perfect $0$-soft projections $p_\alpha:X\to X_\alpha$.
\end{itemize}
\end{thm}
\begin{proof}
Let $X$ be an $\rm{AE}_p(0)$-space of weight $\tau$ and embed $\omega X$ in the Tychonoff cube $\mathbb I^\tau$. By \cite{ho}, $\mathbb I^\tau$ is an image of the Cantor cube $D^\tau$ under a perfect $0$-invertible map $f_0^\tau$. As in the proof of Theorem 3.2, take a closed $G_\delta$-set
$F\subset\mathbb I^\tau$ with $F\cap\omega X=\{\omega\}$ and let $Y=D^\tau\backslash (f_0^\tau)^{-1}(F)$. Then $Y$, as a locally compact Lindel\"{o}f subset of $D^\tau$, is an $\rm{AE}(0)$-space. Indeed, there is a locally compact subset $Y_0$ of the Cantor set $D^{\aleph_0}$ with $\pi^{-1}(Y_0)=Y$, where $\pi:D^\tau\to D^{\aleph_0}$ is the projection. Since $\pi$ is $0$-soft and $Y_0\in\rm{AE}(0)$, $Y\in\rm{AE}(0)$. Next, consider a proper map $g:Y\to X$ extending the restriction $f_0^\tau|(f_0\tau)^{-1}(X)$. Then $g$ is also $0$-invertible, hence $X\in\rm{AE}(0)$ because $Y\in\rm{AE}(0)$. Therefore, $(i)\Rightarrow (ii)$. The implication $(ii)\Rightarrow (iii)$ is well known, see \cite[Proposition 3.9]{chi2}. For the  implication
$(iii)\Rightarrow (iv)$, observe that Haydon's \cite{ha} spectral characterization of compact $\rm{AE}(0)$-spaces implies that $\omega X$ is the limit of a $\sigma$-complete inverse system  $\displaystyle \widetilde S=\{\widetilde X_\alpha, \widetilde p^{\beta}_\alpha\}$ consisting of
compact metric spaces $\widetilde X_\alpha$ and perfect $0$-soft projections $\widetilde p_\alpha:X\to X_\alpha$. Since $\{\omega\}$ is a $G_\delta$-subset of $\omega X$, the restriction of $\widetilde S$ over $X$ provides a $\sigma$-complete inverse system  $\displaystyle S=\{X_\alpha, p^{\beta}_\alpha\}$ consisting of
locally compact separable metric spaces $X_\alpha$ and perfect $0$-soft projections $p_\alpha:X\to X_\alpha$ such that $X$ is the limit of $S$ (see the proof of Theorem 3.2). The implication $(iv)\Rightarrow (i)$ follows from Lemma 3.1 and Proposition 2.3 because $X$ admits a $0$-soft map into a separable locally compact metric space $X_\alpha$.
\end{proof}

\end{document}